\documentclass{article}
\usepackage[utf8]{inputenc}
\usepackage[T1]{fontenc}
\usepackage[margin=1.2in]{geometry}
\usepackage{amsmath}
\usepackage{amsfonts}
\usepackage{amssymb}
\usepackage{amsthm}
\usepackage{mathtools}

\newtheorem{theorem}{Theorem} 
\newtheorem{lemma}{Lemma}

\newtheorem{claim}[lemma]{Claim}

\newtheorem{remark}[lemma]{Remark}
\newtheorem{proposition}[lemma]{Proposition} 

\newcommand*{\myproofname}{Proof}
\newenvironment{claimproof}[1][\myproofname]{\begin{proof}[#1]}{\end{proof}}

\newcommand*{\floorfrac}[2]{\mathopen{}\left\lfloor\frac{#1}{#2}\right\rfloor\mathclose{}}

\usepackage{graphicx}
\usepackage{pgf,tikz,subcaption}
\usetikzlibrary{arrows,shapes}
\usepackage{enumitem}
\usepackage[normalem]{ulem}
\usepackage{hyperref}
\hypersetup{colorlinks = true, linkcolor = blue, citecolor = blue, urlcolor = blue}

\newcommand{\h}{\mathcal{H}}

\title{The maximum Wiener index of a uniform hypergraph}
\author{
Stijn Cambie\thanks{Extremal Combinatorics and Probability Group (ECOPRO), Institute for Basic Science (IBS), Daejeon, KR, supported by the Institute for Basic Science (IBS-R029-C4),
E-mail: {\tt stijn.cambie@hotmail.com, salianika@gmail.com}},
Ervin Gy\H{o}ri\thanks{Alfr\'ed R\'enyi Institute of Mathematics, Budapest, HU, supported by the National Research, Development and Innovation Office NKFIH, grants K132696, K135800 and SNN-135643,
E-mail: {\tt gyori.ervin@renyi.hu,ctompkins496@gmail.com}}, 
Nika Salia\footnotemark[1],
Casey Tompkins\footnotemark[2], 
James Tuite\thanks{Open University, Walton Hall, Milton Keynes, UK, supported by EPSRC grant EP/W522338/1 and London Mathematical Society grant ECF-2021-27.
E-mail: {\tt james.tuite@open.ac.uk}}
}

\begin{document}

\maketitle
\begin{abstract} 
The Wiener index of a (hyper)graph is calculated by summing up the distances between all pairs of vertices.
We determine the maximum possible Wiener index of a connected $n$-vertex $k$-uniform hypergraph and characterize for every~$n$ all hypergraphs attaining the maximum Wiener index.
\end{abstract}

\section{Introduction}

The Wiener index~\cite{Wie47}, also known as the total distance, is a well-studied parameter in graph theory. 
For a given graph $G$, the Wiener index of the graph is denoted by $W(G)$ where
\[
W(G)=\sum_{\{u,v\}\subseteq V(G)} d_{G}(u,v).
\]
The Wiener index of a graph with a fixed number of vertices is linearly related to the average distance of the graph, which is itself another important metric parameter in random graph theory. 
Additionally, it is related to parameters of random walks such as cover cost and Kemeny's constant~\cite{GW17,JSS22}.
Its generalization of Steiner-Wiener index~\cite{LMG16} is related to the average traveling salesman problem distance~\cite{C22}.

It is natural to investigate the extremal properties of the Wiener index. 
Among the most basic results on the total distance of a graph, there are folklore results that for every connected graph $G$ and a tree $T$ with $n$ vertices, we have 
\[
W(K_n) \le W(G) \le W(P_n) \mbox{~and~} W(S_n) \le W(T) \le W(P_n).
\]
Here $K_n$ is the $n$-vertex complete graph, $P_n$ is a path with $n$ vertices and $S_n$ is a star with $n$ vertices. 
Proofs can be found in~\cite{cambie2022thesis}.

While these statements and proofs might be straightforward and easy, the relationship between the Wiener index and other graph parameters or specific classes of graphs is still being explored.
For example among all graphs with a given diameter and order the minimum Wiener index is determined in~\cite{plesnik1984sum}.
In~\cite{cambie2020asymptotic} the maximum Wiener index is determined up to the asymptotics of the second-order term. 
The authors in~\cite{che2019upper,szekely2021wiener,ghosh2020maximum} have studied planar graphs and determined sharp bounds for the maximum of the Wiener index in Apollonian planar graphs and maximal planar graphs.

In this work, we study the Wiener index of $k$-uniform hypergraphs. 
Generalizing fundamental results in graph theory to hypergraphs gives a broader insight into the problem. 
A few examples of such extensions of classical results towards the hypergraph setting are given in~\cite{Mubayi06,GL12}. 
A $k$-uniform hypergraph $\h$ of order $n$ can be represented as a pair $(V,E)$ where $V$ is a vertex set with $n$ elements and $E \subseteq \binom{V}{k}$ is a family of $k$-sets, called hyperedges or just edges. 
Note that in the case of $k=2$, a $2$-uniform hypergraph is a graph.
In order to introduce a distance between the vertices of a hypergraph, at first we need to introduce a concept of paths for hypergraphs. 
In this work we follow the celebrated definition of Berge~\cite{berge1973graphs}: 
a Berge path of a hypergraph $\h$ is a sequence of distinct vertices and hyperedges $v_0 h_1 v_1 h_2 v_2 \ldots v_{d-1} h_d v_d$ of $\h$ where $v_{i-1},v_{i} \in h_i$ for every $1 \le i \le d$ for some $d$. 
A hypergraph is connected if there is a Berge path for every pair of vertices.
The length of a Berge path is the number of hyperedges in the Berge path.
Naturally the distance between two vertices $u,v$ is the length of the shortest Berge path containing $u$ and $v$ in the hypergraph $\h$, and it is denoted by $d_{\h}(u,v)$. 
Note that when the ground hypergraph is known from the context we write $d(u,v)$ instead of $d_{\h}(u,v)$.
The Wiener index of a connected hypergraph $\h$ is defined as $\sum_{\{u,v\} \subseteq V(\h)} d_{\h}(u,v).$

A connected hypergraph $\h$ is called a hyper-tree when it can be represented by a tree $T$ in the graph sense,  on the same vertex set, where every edge of $\h$ induces a connected subgraph of $T$, equivalent formulations are collected in~\cite[Chap.~5, Sec.~4]{Berge89}.
When restricting to linear $k$-uniform hyper-trees (connected hypergraphs with $n=m(k-1)+1$ vertices and $m$ edges), it was shown in~\cite{GZL17} that for every such hyper-tree $\h$, $W(S_n^k) \le W(\h)\le W(P_n^k)$, where $S_n^k$ and $P_n^k$ are the linear path and star, see Figure~\ref{fig:loosepath&star}.
Furthermore, the corresponding extremal hyper-trees are unique.

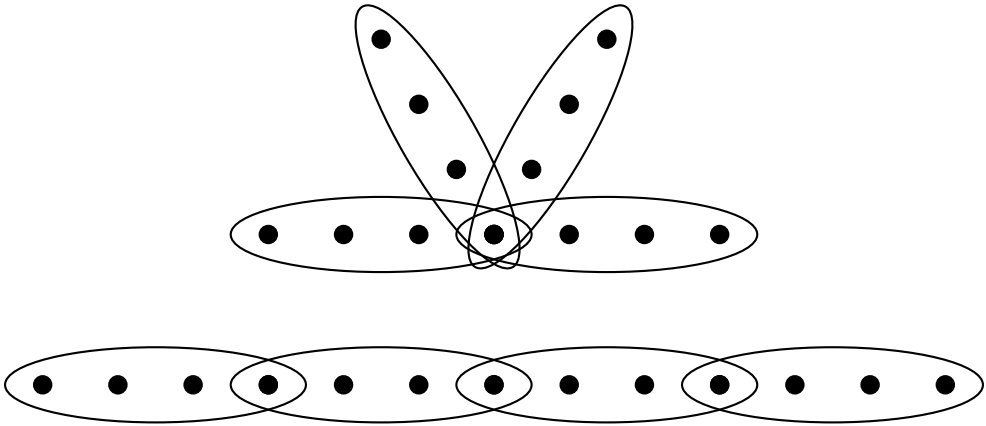
\begin{figure}[h]
 \centering
\begin{tikzpicture}[thick]
\foreach \y in {0,1,2,3}{
\foreach \x in {0, 1,2,3}{
\draw[fill] (60*\y:\x) circle (0.1151);
}
\draw[rotate=60*\y] (0:1.5) ellipse (2cm and 0.5cm);
}

\foreach \y in {0,1,2,3}{
\foreach \x in {0, 1,2,3}{
\draw[fill] (3*\x+\y-6,-2) circle (0.1151);
}
\draw (-4.5+3*\y,-2) ellipse (2cm and 0.5cm);
}

\end{tikzpicture}
 \caption{The loose star $S_{13}^4$ and path $P_{13}^4$}
 \label{fig:loosepath&star}
\end{figure}

Considering arbitrary connected $k$-uniform hypergraphs, or hyper-trees, the extremal hypergraphs for maximizing/minimizing the Wiener index are not always unique. Determining the lower bound for the Wiener index of $k$-uniform hypergraphs is trivial, as all distances between the distinct vertices are at least $1$.

\begin{proposition}\label{prop:lowerbound}
 Any connected $k$-uniform hypergraph $\h$ satisfies
 $\binom{n}{2}=W(K_n^k) \le W(\h).$
 %Among hypetrees, equality is attained if and only if the underlying tree $T$ has diameter less than $k.$
\end{proposition}

We remark that, in huge contrast with the graph case, equality can also be attained by very sparse hypergraphs.

\begin{remark}
Equality in Proposition~\ref{prop:lowerbound} occurs if and only if for every pair of vertices there is a hyperedge incident to both. This is the case when the hypergraph corresponds with a projective plane. 
This is an example of a sparse hypergraph minimizing the Wiener index. Also hyper-trees can attain equality when $k \ge 3$, e.g. the hyper-tree with underlying tree $S_n$ and every subtree of order $k$ being a hyperedge.
\end{remark}

On the other hand, maximizing the Wiener index for $k$-uniform hypergraphs is a relatively complex problem.
In this note, we determine the maximum possible Wiener index of a connected $n$-vertex $k$-uniform hypergraph and characterize all hypergraphs attaining the maximum.
First, we define the class of extremal hypergraphs.
For integers $a$ and $b$ with $a\leq b$, let $[a, b]$ be the set $\{a,a+1,\ldots,b\}$ and $[a]$ be the set $[1, a]$.
For given integers $n$ and $k$ with $0\le k<n$, let $s$ and $r$ be integers such that $n=ks+r$ and $0\leq r<k$.
For every $r$ such that $r\neq 0$, let $P_{n}^{k}=(V,E)$ be the following tight-path. 
Let the vertex set of $P_n^{k}$ be $[n]$ and the edge set be
$ E=\left \{[(i-1)k+1, ik] \colon 1 \le i \le s \right\} \cup \left\{ [r+(i-1)k+1, ik+r] \colon 1 \le i \le s\right \}.$ 
%see Figure~\ref{fig:densepath}.
For $r=0$, and every integer $x$, such that $0<x<k$, we define a $k$-uniform hypergraph $P_{n,x}^{k}=(V,E)$ where $V=[n]$ and
$E=\{ [(i-1)k+1, ik] \colon 1 \le i \le s\} \cup \{ [(i-1)k+1+x, ik+x] \colon 1 \le i \le s-1\}.$ See Figure~\ref{fig:densepath} for examples of both $P_n^{k}$ and $P_{n,x}^{k}$.
Our main result can now be stated as follows.

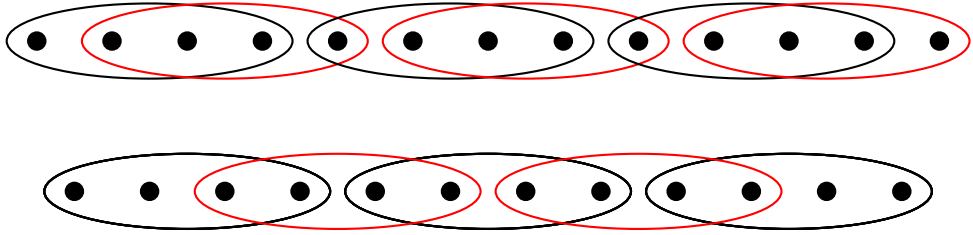
\begin{figure}[h]
 \centering
\begin{tikzpicture}[thick]

\foreach \y in {0,1,2}{
\foreach \x in {0, 1,2,3}{
\draw[fill] (\x+4*\y,-2) circle (0.1151);
}
\draw (1.5+4*\y,-2) ellipse (1.9cm and 0.5cm);
\draw[red] (2.5+4*\y,-2) ellipse (1.9cm and 0.5cm);
}
\draw[fill] (12,-2) circle (0.1151);

\foreach \y in {0,1,2}{
\foreach \x in {0, 1,2,3}{
\draw[fill] (\x+4*\y+0.5,-4) circle (0.1151);
\draw (2+4*\y,-4) ellipse (1.9cm and 0.5cm);
}
}
\foreach \y in {0,1}{
\draw[red] (4+4*\y,-4) ellipse (1.9cm and 0.5cm);
}

\end{tikzpicture}
 \caption{The tight paths $P_{13}^4$ and $P_{12,2}^4$ for $k=4$}
 \label{fig:densepath}
\end{figure}

\begin{theorem}\label{thr:main}
 Let $\h$ be a connected $k$-uniform hypergraph of order $n\geq k$.
 If $k \nmid n,$ then $W(\h) \le W(P_n^{k})$ with equality if and only if $\h=P_n^{k}.$
 If $k \mid n$, then $W(\h) \le W(P_{n,1}^{k})$ with equality if and only if $\h=P_{n,x}^{k}$ for some $0<x<k.$
\end{theorem}
%This theorem will be proven in Section~\ref{sec:proofmain}. The proof uses induction and a lemma showing that one can select an edge whose removal does not disconnect the remaining hypergraph (apart from the vertices of degree one belonging to that edge).

\section{Proof of Theorem~\ref{thr:main}}\label{sec:proofmain}
Let $\h=(V,E)$ be an $n$-vertex $k$-uniform connected hypergraph with maximum Wiener index.
Then we may assume that for every edge $h$ of $\h$, the hypergraph $\h'=(V,E \backslash h)$\footnote{We abuse notation and write $E \backslash h$ instead of $E \backslash \{h\}$.} is not connected, otherwise we would remove the edge and the Wiener index would not decrease.
Proving Theorem~\ref{thr:main} for edge-minimal hypergraphs is sufficient since adding an edge to one of the extremal hypergraphs from Theorem~\ref{thr:main} decreases the Wiener~index.
The following lemma shows that there is a hyperedge $h$ in $E$, such that hypergraph $\h'=(V,E \backslash h)$ an contains at most one connected component of size greater than one.

\begin{lemma}\label{lem:findgoodedge}
 Let $\h=(V,E)$ be a $k$-uniform hypergraph such that the deletion of any edge disconnects the hypergraph.
 Then there is an edge $h$ of $\h$ such that the hypergraph $\h'=(V,E \backslash h)$ contains at most one connected component of size greater than one.
\end{lemma}

\begin{proof}
Let $h$ be a hyperedge of $\h$ such that the size of the largest connected component of $\h'=(V,E \backslash h)$ is the maximum.
Then either we are done or $\h'=(V,E \backslash h)$ contains two connected components $\h_1$ and $\h_2$ each containing a hyperedge such that $\h_1$ is a component with the maximum size. 
Let $h_2$ be a hyperedge of $\h_2$, the hypergraph $(V,E \backslash h_2)$ contains a component of size larger than $H_1$ since it contains a component containing $\h_1$ with all vertices of $h$ as well, a contradiction.
%while the vertices of the hyperedge $h_2$ and hypergraph $H_1$ are disjoint
\end{proof}
%If the lemma is not true, assume $\h$ is a smallest counterexample to the statement. Then for a fixed hyperedge $h$, $(V,E \backslash h)$ contains at least two components of cardinality at least two. Let $\h_1$ be one such component. Since $\h$ was a smallest counterexample, $\h_1$ satisfies the condition. If $\h_1$ would consist of a single edge $h_1,$ $\h'=(V,E \backslash h_1)$ contains only one component of cardinality at least two. In the other case $\h_1$ has at least two edges, $h_1$ and $h_2$ and without loss of generality we can assume that $h_2\cap h$ is not empty (some edge of $\h_1$ intersects $h$). But since all edges in $E \backslash E(\h_1)$ are connected, the edges of $\h_1 \backslash h_1$ are connected and these two have at least $h_2\cap h$ in common, we conclude that $\h'=(V,E \backslash h_1)$ again contains only one component of cardinality at least two. This proves the lemma.

By Lemma~\ref{lem:findgoodedge} and the edge-minimality of $\h$, there is an edge $h_1$ of $\h$ such that $\h'=(V,E \backslash h_1)$ consists of a connected component of order $n-\ell$ on a vertex set $V'\subset V$ and $\ell$ isolated vertices for some $1\leq \ell <k$.
Let $\h'_{V'}$ be the largest connected component of $\h'$.

Now we proceed by induction, noting that the theorem trivially holds for $n<2k$. 
Assume Theorem~\ref{thr:main} holds for every order strictly smaller than $n$. 
Let $v$ be an isolated vertex of $\h'$ such that $\sum_{u \in V'} d_{\h'}(u,v)$ is the maximum possible.
Then we have
\begin{equation}\label{eq:inductionstep}
W(\h) \le W(\h'_{V'}) + \binom{\ell}{2} + \ell \sum_{u \in V'} d(u,v). \end{equation}
By induction, we have an upper bound on $W(\h'_{V'}).$
For this, we compute the Wiener index of the paths $P_n^k$ and $P_{n,x}^k$. 
When $n=ks+r,$ for some integers $s$ and $r$ such that $0\leq r<k$, the Wiener index of the path equals
\begin{align*}
 f(s,k,r)&=\frac{k^2s^3}3 + rks^2 +r^2s+ \frac{(k^2 - 3k)s}6 + \binom{r}{2},
\end{align*}
Note that the Wiener index of $P_{n,x}^k$ is independent of the choice of $x$.

We also need to find an upper bound for $\sum_{u \in V'} d_{\h}(u,v)$. To this end, we define the following two functions:
\begin{align*}
 g_1(\ell,k,s,r)&=k s^2 + \ell s+ r(2s+1),\\
 g_2(\ell,k,s,r)&=k s^2 + \ell s + 2rs.\\
\end{align*}

\begin{claim}
 Let $n-\ell=ks'+r'$ for integers $s'$ and $r'$ for $0\leq r'<k$, then we have 
 \[
 \sum_{u \in V'} d(u,v)\le \begin{cases}
 g_1(\ell,k,s',r') \mbox{ if } k-\ell \le r'\\
 g_2(\ell,k,s',r') \mbox{ if } k-\ell > r'\\
 \end{cases}
 \]
\end{claim}

\begin{claimproof}
 Let $d$ be the eccentricity of $v$, i.e. $d=\max_{u \in V'} d_{\h}(u,v).$
 For $d\ge i \ge 1$, let $n_i$ be the number of vertices in $\h'$ which are at distance $i$ from $v$.
 Let $n_i=0$ for $i>d.$
 Now by definition, we have $\sum_{u \in V'} d(u,v)= \sum_{i=1}^d i n_i$ and $\sum_{i=1}^d n_i=(n-\ell)=ks'+r'$.
 Note that $n_1\ge k-\ell$ and $n_i+n_{i+1}\ge k$ for every $1\le i\le d-1,$ since for every such $i$ there is an edge whose vertices are all at distance $i$ and $i+1$ from $v.$
 
 First, we bound the eccentricity.
 Let $u$ be a vertex for which $d(u,v)=d$ and $h$ an edge containing the vertex $u$.
 When $0 \le r' < k-\ell,$ either $d\le 2s'$ or we have $n_1+\dots +n_{2s'-1} \ge (k-\ell)+(s'-1)k>k(s'-1)+r'$ and hence $h$ cannot be disjoint of these vertices at distance at most $2s'-1$ from $v$, thus we have $d\le 2s'.$ Now we conclude 
 \begin{align*}
 \sum_{i=1}^d i n_i
 %&=d(n-\ell) - \sum_{j=1}^{d-1} \left( \sum_{i=1}^{j} n_i \right)\\
 &= 2s' \left( \sum_{i=1}^{2s'} n_i \right) - \sum_{j=1}^{2s'-1} \left( \sum_{i=1}^{j} n_i \right)\\
 &\le 2s'(n-\ell) - \left( s'(k-\ell) + \sum_{j=1}^{2s'-1} \floorfrac j2 k \right)\\
 &= 2s'(ks'+r') -\left( s'(k-\ell) + s'(s'-1)k \right)\\
 &= s'^2k+2s'r'+\ell s'.
 \end{align*}
 In the case that $k-\ell \le r' <k,$ we similarly have $n_1+\dots +n_{2s'} \ge ks' $ and again $h$ cannot be disjoint from at least $ks'$ vertices in $\h'$, thus $d\le 2s'+1.$
 The computation now analogously results into $\sum_{i=1}^d i n_i\le s'^2k+(2s'+1)r'+\ell s'.$
\end{claimproof}
Finally, it is sufficient to prove that the upper bound for the right-hand side of Inequality~\eqref{eq:inductionstep} is always bounded by the claimed upper bound.
For $n-\ell=ks+r$, we have
\begin{align*}
 f(s+1,k,r+\ell-k)&-\left( f(s,k,r)+\ell g_1(\ell,k,s,r)+\binom \ell 2\right)=(k-\ell-r)^2\geq 0 \: \: \mbox{ if } k-\ell \le r,\\
 f(s,k,r+\ell)&-\left( f(s,k,r)+\ell g_2(\ell,k,s,r)+\binom \ell 2\right)=\ell r \geq 0\: \: \:\quad \quad \quad \quad \mbox{ if } k-\ell > r,
\end{align*}
which also has been verified by Maple\footnote{See \url{https://github.com/StijnCambie/WienerHypergraph}}.
%Equality holds if and only if $k=\ell+r$ for any $0<\ell<k$ and thus $r>0$, which exactly corresponds with the extremal graphs $P_{n,\ell}^{k}$ when $k \mid n,$ or $r=0$ when $k \nmid n$ in which case equality is only possible by $P_n^{k}$.

Equality holds if $k=\ell+r$ or $r=0$. 
If $k \mid n,$ equality occurs if and only if $k=\ell+r$ and for every fixed $0<\ell<k$ the only hypergraph attaining the maximum Wiener index is $P_{n,\ell}^{k}$.
If $k\nmid n,$ equality occurs if and only if $r=0$, since $0<\ell<k$ and by the induction hypothesis the unique connected $n$-vertex $k$-uniform hypergraph maximizing the Wiener index is $P_n^{k}$. 
\bibliographystyle{abbrv}
\bibliography{WH}

\end{document}